\pdfoutput=1 
\documentclass[12pt]{amsart}

\usepackage{epsfig}\usepackage{amsmath}
\usepackage{amsfonts}

\newcommand{\Lie}{\text{Lie}}
\newcommand{\ric}{\text{ric}}
\newcommand{\Spin}{\text{Spin}}

\newcommand{\fp}{\mathfrak{p}}

\newcommand{\M}{\mathcal{M}}

\newcommand{\dt}[0]{\frac{d}{dt}}

\usepackage[T1]{fontenc}
\usepackage{amsmath, amsfonts, amsthm, amssymb,transparent}
\usepackage[mathscr]{eucal}		

\makeindex
\usepackage{caption}
\usepackage[normalem]{ulem}	

\usepackage{amsmath}
		\usepackage{amsfonts}
		\usepackage{amssymb}

\allowdisplaybreaks[3]
\usepackage{graphicx}
\usepackage{geometry}

\numberwithin{equation}{section}

\usepackage{booktabs}

\usepackage{tikz}
\tikzset{every picture/.style={line width=1}}
\pgfmathsetmacro{\ptsize}{0.04}   

\theoremstyle{definition}
\newtheorem{thm}{Theorem}[section]
\newtheorem{prop}[thm]{Proposition}
\newtheorem{main}{Theorem}

\newtheorem{lem}[thm]{Lemma}

\theoremstyle{definition}

\newcommand{\tg}{\tilde{g}}

\newcommand{\RR}{\mathbb{R}}

\newcommand{\HH}{\mathbb{H}}

\newcommand{\CCPP}{\mathbb{CP}} 	  
\newcommand{\Sp}{\text{Sp}}

\newcommand{\f}[1]{\mathfrak{#1}}

\renewcommand{\bar}[1]{\overline{#1}}

\renewcommand{\tilde}[1]{\widetilde{#1}}

\renewcommand{\phi}{\varphi}

\newcommand{\rvline}{\hspace*{-\arraycolsep}\vline\hspace*{-\arraycolsep}}

\let \Re \relax

\DeclareMathOperator{\Re}{Re}

\DeclareMathOperator{\Id}{Id}
\let \1 \relax
\DeclareMathOperator{\1}{\textbf{1}}

\DeclareMathOperator{\Ad}{Ad}

\DeclareMathOperator{\fu}{\f{u}}

\DeclareMathOperator{\fg}{\f{g}}
\DeclareMathOperator{\fh}{\f{h}}

\DeclareMathOperator{\fsp}{\f{sp}}

\DeclareMathOperator{\uSO}{\textup{SO}}
\DeclareMathOperator{\uU}{\textup{U}}

\newcommand{\Rm}{\text{Rm}}
\newcommand{\Ric}{\text{Ric}}

\newcommand{\sst}[0]{\;:\;}

\begin{document}
	\title{On the Ricci flow of homogeneous metrics on spheres}
	\author{Sammy Sbiti}
	\date{\today}
		\begin{abstract}	
		We study the Ricci flow of the four-parameter family of $\Sp(n+1)$-invariant metrics on  $S^{4n+3}\subset\mathbb{R}^{4n+4}$. We determine their forward behaviour and also classify ancient solutions. In doing so, we exhibit a new one-parameter family of ancient solutions on spheres. 
	\end{abstract}
	\maketitle

Hamilton's Ricci flow is given by the geometric PDE
$$
	\frac{\partial}{\partial t}g_t=-2\Ric(g_t).
$$
Similar to the heat equation, the Ricci flow has regularizing properties for Riemannian metrics, making it useful for proving classification-type theorems in geometry. It was used, for example, to prove that simply connected $3$-manifolds with positive Ricci curvature are spheres \cite{hamilton1982three}, as well as simply connected $n$-manifolds with positive curvature operators \cite{bohm2008manifolds} and those with quarter-pinched metrics \cite{brendle2009manifolds}. The Ricci flow with surgery was used in Perelman's celebrated proof of the Poincar\'{e} conjecture, and was also used in recent work by S. Brendle to classify manifolds with positive isotropic curvature in dimensions $n\ge 12$ \cite{brendle2019ricci}. 
 The normalized Ricci flow preserves volume and can be obtained from the Ricci flow by a rescaling and reparametrization in time (see Section 1).
On a compact homogeneous space $M^n=G/H$ the Ricci flow is equivalent to a system of ODE's. In particular, the backwards flow is always well-posed. Moreover, the normalized flow is the gradient flow for the scalar curvature functional on the space of $G$-invariant volume-1 metrics on $M$, and the critical points are Einstein metrics (see, e.g. \cite{wang1986existence}). The global behaviour of the scalar curvature functional was studied in \cite{bohm2004variational}, where it was shown that it satisfies a certain compactness condition, which in particular implies that the moduli space of $G$-invariant Einstein metrics is compact. 
C. B\"{o}hm proved that if $G/H$ is compact and not a torus, then the Ricci flow develops a Type-1 singularity in finite time \cite{bohm2015long}. Moreover, under some mild conditions, a sequence of parabolic rescalings coverges to $K/H\times\mathbb{R}^{n-k}$, for some intermediate subgroup $H\le K\le G$, where $K/H$ is endowed with an Einstein metric and $\RR^{n-k}$ is flat \cite{bohm2015long}.

The goal of this article is to describe the Ricci flow for homogeneous metrics on spheres and to classify their ancient solutions. Besides the left-invariant metrics on $S^3$, homogeneous metrics on spheres can be described in terms of the Hopf fibrations \cite{ziller1982homogeneous} \begin{equation}
S^1\to S^{2n+1}\to \mathbb{CP}^n \qquad
S^7\to S^{15}\to S^8\qquad
S^3\to S^{4n+3}\to \mathbb{HP}^n \label{Hopf}.
\end{equation}

Associated to each fibration, there is a $2$-parameter family of homogeneous metrics $g_{t,s}$ which can be obtained by starting with the round metric and scaling the fiber by $t$ and the base by $s$. The behaviour of the Ricci flow for these metrics was studied in \cite{buzano2014ricci} and \cite{bakas2012ancient}, and we indicate their behaviour in Section 1. See also \cite{isenberg1992ricci} and \cite{cao2008backward} for the case of left-invariant metrics on $S^3$.
There exists, however, a larger class of homogeneous metrics associated to the fibration $S^3\to S^{4n+3}\to\mathbb{HP}^n$ by allowing the metric on $S^3$ to be an arbitrary left-invariant metric. It can be shown that, up to isometry, this left-invariant metric can be diagonalized with eigenvalues $x\le y\le z$. Hence, we obtain a $4$-parameter family of metrics, which we denote by $g_{x,y,z,s}$. This is the only family of homogeneous metrics on spheres for which the Ricci flow has not yet been studied, and is the main object of this paper.

In \cite{ziller1982homogeneous}, it was shown that the only $\Sp(n+1)$-invariant Einstein metrics on $S^{4n+3}$ are, up to scaling, the round metric, $g_{\text{rd}}$, and Jensen's second Einstein metric, $g_{E_2}=g_{x,s}$ where $(x,s)=(1,2n+3)$ \cite{jensen1973einstein}. Hence, the normalized Ricci flow has two fixed points, or so-called nodes. For the normalized flow, a solution either converges to an Einstein metric or the scalar curvature goes to infinity. It follows from a theorem in \cite{bohm2015long} that in the latter case a suitable blow-up converges to an isometric product $S^3\times \mathbb{R}^{4n}$ where $S^3$ and $\mathbb{R}^{4n}$ are endowed with the round metric and flat metric, respectively. We give an elementary proof and in addition exhibit some qualitative behaviours for its solutions.
\begin{main}Jensen's second Einstein metric has a two-dimensional stable manifold for the normalized Ricci flow, which separates the space of unit-volume $\Sp(n+1)$-invariant metrics into two connected, invariant components. Any solution in the first component converges to the round metric in infinite time, and for any solution in the second component, a suitable blow-up converges to $S^3\times \mathbb{R}^{4n}$ where $S^3$ and $\mathbb{R}^{4n}$ are endowed with the round metric and the flat metric, respectively.
\end{main}
 Furthermore, along the flow $x,y,z\to 0$, and if we assume $x\le y\le z$ then $x/z$ and $y/z$ converge monotonically to $1$.
 
Recall that a solution of the Ricci flow $g_t$ is called ancient if it exists for $t\in(-\infty,0]$. Ancient solutions arise by blowing up near a singular time by a parabolic sequence of rescalings, so understanding them is crucial for singularity analysis. For some examples of ancient solutions in dimensions $n\ge 3$, see e.g. \cite{bohm2019optimal,lauret2013ricci,lu2017ancient,cao2008backward,fateev1996sigma,perelman2002entropy,brendle2017gluing}. As we will see, a solution of the Ricci flow on a compact homogeneous space is ancient if and only if the corresponding solution for the normalized flow is also ancient. We classify the ancient $\Sp(n+1)$-invariant solutions of the Ricci flow on $S^{4n+3}$ and exhibit a new one-parameter family.
\begin{main}
	Let $g_t$ be an $\Sp(n+1)$-invariant solution of the Ricci flow or the normalized Ricci flow on $S^{4n+3}$ with initial condition $g_0=g_{x,y,z,s}$, where $x\le y\le z$. Then $g_t$ is ancient if and only if  $x\le y=z\le s$.
\end{main}

These (non-isometric) ancient solutions all have a larger isometry group, namely $\Sp(n+1)\Sp(1)$, $\Sp(n+1)\uU(1)$, or $\uU(2n+2)$ (see Theorem \ref{U(1)}). Two ancient solutions converge, under the backwards flow, to Jensen's second Einstein metric and are non-collapsed (see \cite{bakas2012ancient}). All the remaining ones are collapsed and can be viewed as shrinking the fiber of the Hopf fibration $S^1\to S^{4n+3}\to\mathbb{CP}^{2n+1}$ and simultaneously letting the metric on the base vary. One solution parametrizes the well known Berger metrics. The rest of the solutions are new. If one rescales appropriately, then, under the backwards flow, these solutions converge in the Gromov-Hausdorff topology to Ziller's second homogeneous Einstein metric on $\mathbb{CP}^{2n+1}$ \cite{ziller1982homogeneous}. Similar to the ancient solutions found in \cite{bohm2019optimal}, the limit solitons do not depend continuously on the initial metric.
Figure \ref{fig: ancient4} illustrates the behaviour of the backwards flow for the volume-normalized solutions, which can be obtained by setting $x=\frac{1}{y^2 s^{4n}}$. The $y$-axis represents the ratio $y/s$ and the $x$-axis represents the value of $s$. See also Figure \ref{fig: two-parameter families} and Figure \ref{Scal1} for the graph of the scalar curvature functional on the set of volume 1 metrics.
 \begin{figure}[h]

 		\resizebox{163mm}{!}{\includegraphics{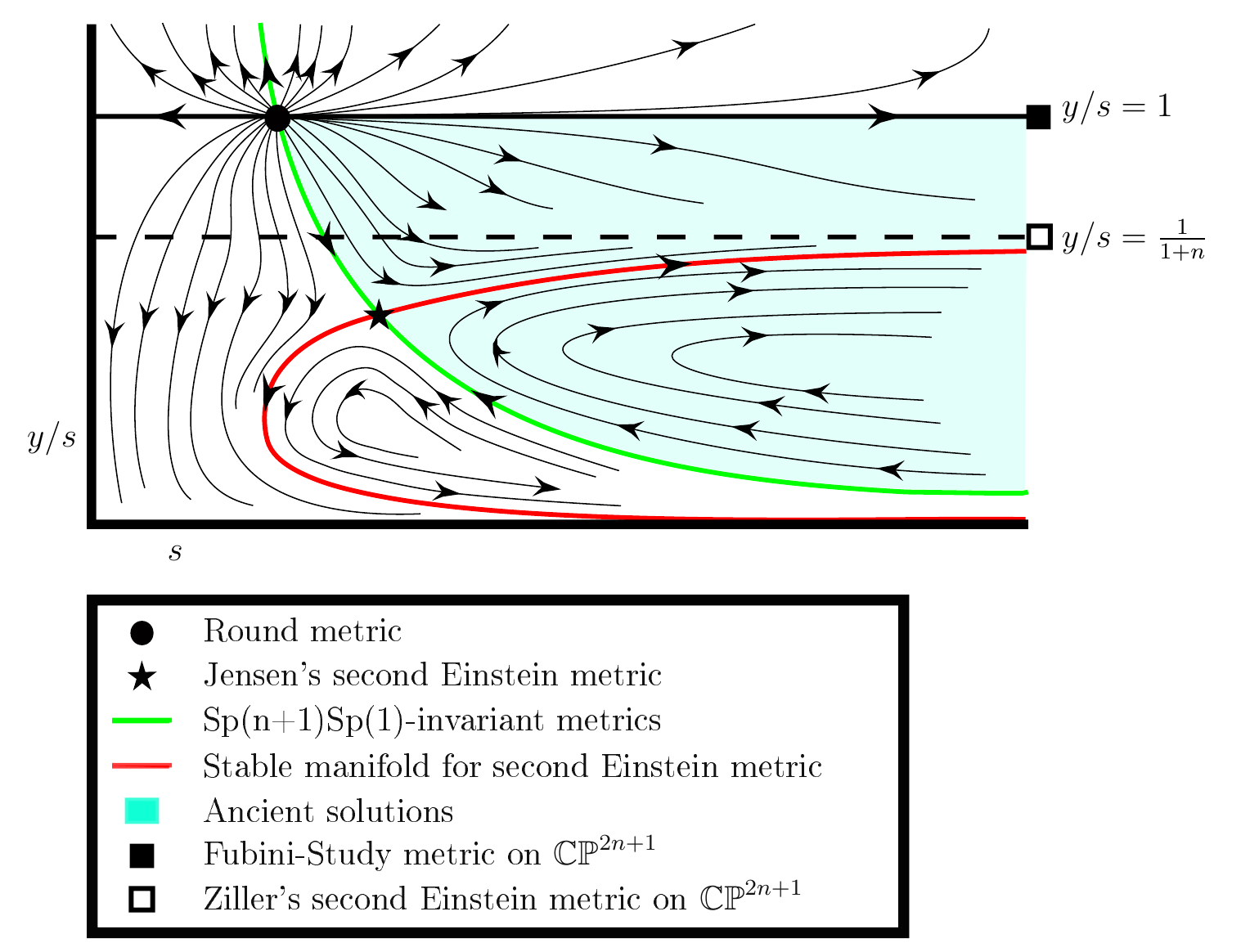}}
 		 	\centering{
 		\caption{Backwards flow of $\Sp(n+1)\uU(1)$-invariant metrics}
 		\label{fig: ancient4}}
 \end{figure}

The paper is organized as follows. In Section $1$ we discuss known properties of the Ricci flow on homogeneous spaces and describe the relevant homogeneous metrics on spheres in more detail. In Section $2$ we relate the existence-time of the Ricci flow on any compact homogeneous space with that of the normalized flow. Here we prove that being ancient for the Ricci flow is equivalent to being ancient for the normalized flow, and that the normalized flow develops a finite-time singularity unless it converges to an Einstein metric. In Section $3$ we study the dynamics of the forwards Ricci flow for $\Sp(n+1)$-invariant metrics and prove Theorem A. In Section $4$ we study the backwards flow and prove Theorem B.

I would like to thank my PhD advisor Wolfgang Ziller for his many helpful comments, discussions, and suggestions. Part of this research was conducted at IMPA, and I am grateful for their hospitality.
\section{Preliminaries}
	
	In the remainder of the paper we alternate between the notation $g_t$ and $g(t)$ for a solution of the Ricci flow, and between $\tg_t$ and $\tg(t)$ for a solution of the normalized flow, whenever convenient. All manifolds and homogeneous spaces are assumed to be compact.
	
	We denote by $\M$ the space of Riemannian metrics on the manifold $M^n$ and $\M^G\subset\M$ the space of $G$-invariant metrics, where $G$ is a Lie group acting on $M$. Likewise, we denote the space of volume-$1$ metrics on $M$ by $\mathcal{M}_1$ and the space of $G$-invariant volume-$1$ metrics by $\M^G_1\subset\M^G$.

	 The space $\M$ can be endowed with the $L^2$ inner product, which assigns to any symmetric $2$-tensor field $h$ (viewed as a tangent vector to a metric $g\in\M$) the length 
 	 $$||h||^2_g=\int_M g(h,h)d\mu_g,$$
	 where $d\mu_g$ is the volume element for $g$.
	 
	 We denote by $\textbf{S}$ the total scalar curvature functional on $\M$:
	 $$\textbf{S}(g)=\int_M S(g) d\mu_g,$$
	 where $S(g)$ is the scalar curvature of $g$.
	 
	 It is well known that, restricted to $\M_1$, the $L^2$ gradient of $\textbf{S}$ is given by the negative traceless Ricci tensor $\nabla\textbf{S}=-\Ric^0(g)=-(\Ric(g)-\frac{S(g)}{n}g)$ (e.g. \cite{besse2007einstein}, p.120). In particular, Einstein metrics are the critical points of $\textbf{S}$.
	
	 Let $M=G/H$ be a homogeneous space where $H\le G$ are compact Lie groups with Lie algebras $\fh\subset \fg$.
	 Since $G$ is compact we can fix an $\Ad(G)$-invariant inner product $Q$ on $\mathfrak{g}$. Let $\mathfrak{p}$ be the $Q$-orthogonal complement of $\mathfrak{h}$ in $\mathfrak{g}$ so that $\mathfrak{g}=\mathfrak{h}\oplus\mathfrak{p}$. Then via action fields there is an isomorphism $\mathfrak{p}\simeq T_{p}M: X\mapsto\dt|_{t=0}\exp(tX)\cdot p$ where $p\in M$ is the identity coset.
	 Moreover, there is a one-to-one correspondence between $\Ad(H)$-invariant inner products on $\mathfrak{p}$ and $G$-invariant metrics on $G/H$, and hence $\M^G$ is a finite-dimensional manifold.	 
	 
	 Note also that for homogeneous metrics, scalar curvature and Ricci curvature are constant, so, in particular, on $\M^G_1$, the functional $\mathbf{S}$ just assigns to each metric its scalar curvature at a point. From now on, we restrict to $\M^G_1$ and identify $\mathbf{S}(g)$ with $S(g)$, so that the $L^2$ gradient of $S$ is $\nabla S(g)=-\Ric^0(g)$.

	Recall that a solution of the Ricci flow is a smooth family of metrics $g_t\in \M$ satisfying
	
	$$
	\frac{\partial}{\partial t}g_t=-2\Ric(g_t)
	.$$
	Since the Ricci tensor is diffeomorphism invariant, isometry groups are preserved under the Ricci flow. That is, if $g_0\in\M^G$, then $g_t\in\M^G$ for all $t$.

The normalized Ricci flow, which we denote by $\tg_t$, is given by the equation

$$
\frac{\partial}{\partial t}\tg_t=2\left(-\Ric(\tg_t)+\frac{\mathbf{S}(\tg_t)}{n \text{Vol}_{\tg_t}(M)}\tg_t\right).
$$
The normalized Ricci flow preserves volume and can be obtained from the Ricci flow as follows.
Let $g(t)$ be a solution of the Ricci flow and let $S(t):=S(g(t))$. The corresponding solution of the normalized flow is given by $\tilde{g}(f(t))=r(t)g(t)$ where $$r(t)=\exp\left(\frac{2}{n}\int_0^t S(\tau)d\tau\right),\quad f'(t)=\exp\left(\frac{2}{n}\int_0^t S(\tau)d\tau\right),$$ and $f(0)=0$ (see for instance \cite{chow2006hamilton}). Hence we can restrict the normalized flow to $\M^G_1$ where it becomes an ODE given by 
$$
\frac{\partial}{\partial t}\tg_t=2\left(-\Ric(\tg_t)+\frac{S(\tg_t)}{n }\tg_t\right)=2\nabla S(\tg_t).
$$
In particular, the normalized Ricci flow coincides with the $L^2$ gradient flow for $S$.  
 For the remainder of the paper, we denote by $(T_\text{min},T_\text{max})$ the maximal time interval on which the Ricci flow exists, as well as $(\tilde{T}_{\text{min}},\tilde{T}_{\text{max}})$ for the normalized flow.

In \cite{bohm2004variational}, the authors studied the global behaviour of $S$ on $\M^G_1$ with the goal of determining sufficient conditions for the existence of a $G$-invariant Einstein metric. In particular, they proved that, for any fixed $\epsilon>0$, the functional $S$ satisfies the Palais-Smale compactness condition on the set $(\M^G_1)_\epsilon=\{g\in\M^G_1\sst S(g)\ge\epsilon\}$. That is, every sequence of metrics $\{g_i\}_{i=1}^\infty$  in $(\M^G_1)_\epsilon$ with $|S(g_i)|$ bounded and $|\nabla S(g_i)|=|\Ric^0(g_i)|\to 0$ has a convergent subsequence, which, in particular, converges to an Einstein metric. As a consequence, the set of $G$-invariant Einstein metrics has only finitely many components, and each of them is compact. They also noted that this result is optimal in the sense that it is impossible to have a convergent sequence of metrics in $\M^G_1$ with $S(g_i)<0$ and $|\Ric^0|\to 0$ since the limit would have to be an Einstein metric of negative scalar curvature, or would have to be Ricci flat. The first possibility is ruled out by Bochner's theorem, and the second can only occur if $M$ is flat by Alekseevsky-Kimel'fel'd \cite{alekseevskii1975structure}. On the other hand, there may exist sequences $\{g_i\}_{i=1}^{\infty}$ with $S(g_i)>0$, $S(g_i)\to 0$ and $|\Ric^0(g_i)|\to0$, so-called $0$-Palais Smale sequences, which do not converge unless $M$ is a torus.

From now on, we will assume that $M$ is a compact homogeneous space which is not a torus. By Theorem $1$ and Theorem $2$ in \cite{bohm2015long}, a homogeneous solution $g_t$ to the Ricci flow on $M$ develops a Type-1 singularity in finite time. Recall that a finite-time singularity of the Ricci flow is said to be Type-1 if the curvature tensor blows up at most linearly, that is, there exists some $C>0$ so that $|\Rm(g_t)|(T_{\text{max}}-t)\le C$ for $t$ near $T_\text{max}<\infty$, where $|\Rm|$ is the norm of the curvature tensor. By \cite{lafuente2015scalar}, this implies that the scalar curvature goes to $+\infty$ near the singular time. In particular, by starting the flow at a later time, we can assume $S(g_0)>0$. 



Since $\tg_t$ is the gradient flow for $S$, we have

 \begin{equation} d_{L^2}(\tg_t,\tg_0)\le \int_0^t |\Ric^0(\tg_s)|ds\le t^{1/2}\left(\int_0^t |\Ric^0(\tg_s)|^2 ds\right)^{1/2}, \label{gradientflow}
 \end{equation}
  and $$S(\tg_t)-S(\tg_0)=\int_0^t |\Ric^0(\tg_s)|^2 ds.$$ Thus, there are two possibilities for solutions $\tg_t$ of the normalized Ricci flow on $\M^G_1$. The first possibility is that $S(\tg_t)\to\infty$ as $t\to \tilde{T}_\text{max}$, and the second is that $S(\tg_t)\le C$ for all $t\in(0,\tilde{T}_\text{max})$, in which case (\ref{gradientflow}) implies that $\tilde{T}_\text{max}=\infty$. Similarly, (\ref{gradientflow}) implies that if $S$ is bounded from below then $\tilde{T}_\text{min}=-\infty$. In the case that $S(\tg_t)\le C$, Palais-Smale further implies that $\tg_t$ converges to an Einstein metric as $t\to\infty$.


We will also examine solutions of the Ricci flow as $t\to-\infty$. As remarked above, a lower bound on $S$ already implies that $\tg_t$ is ancient. If $g_t$ is an ancient solution of the Ricci flow then it is an easy consequence of the maximum principle applied to the evolution equation for $S$, 
\begin{equation}
S(g(t))'=\Delta S(g(t))+2|\Ric(g(t))|^2,
\end{equation}
that $g_t$ either has positive scalar curvature for all time, or is Ricci flat (e.g. \cite{chow2006hamilton} p. 102).
 Since $G/H$ is not a torus, the latter is ruled out. Hence, the corresponding solution $\tg_t$ of the normalized flow has $S(\tg_t)>0$ for all $t$, which implies that $\tilde{T}_{\text{min}}=-\infty$ and $|\Ric^0(\tg_t)|\to 0$ as $t\to-\infty$. Thus there are two possibilities for the corresponding solution $\tg_t$ of the normalized flow. The first possibility is that $S(\tg_t)>\epsilon>0$ for all time, in which case, by Palais-Smale, $\tg_t$  converges to an Einstein metric as $t\to-\infty$. The second possibility is that $S(\tg_t)\to 0$, i.e., $\tg_t$ is $0$-Palais-Smale.
$0$-Palais-Smale sequences were studied in \cite{bohm2004variational} where the authors showed that if one exists, then there exists an
intermediate subgroup $H\le K\le G$ such that $K/H$ is a torus.
In \cite{pediconi2019diverging}, F. Pediconi further proved that every divergent sequence of metrics in $\M^G_1$ with $|\text{Rm}|$ bounded has a subsequence which asymptotically approaches a submersion metric for a torus fibration with shrinking fibers as in (\ref{fibration-intro}). By the Gap Theorem for compact homogeneous spaces, $|\Rm(g)|\le C|\Ric(g)|$ (see [B\"{o}2] Theorem 4), and hence $0$-Palais Smale sequences are special examples of divergent sequences of metrics with bounded curvature.

Let us recall the definition of submersion metrics.
If $\mathfrak{h}\subset\mathfrak{k}\subset\mathfrak{g}$ is an intermediate Lie subalgebra with $\mathfrak{k}=\Lie(K)$, then we can further decompose $\mathfrak{p}=\mathfrak{p}_\mathfrak{k}+\mathfrak{p}_\mathfrak{k}^\perp$ where $\mathfrak{p}_\mathfrak{k}=\mathfrak{p}\cap\mathfrak{k}$. We say that an $\Ad(H)$-invariant inner product $g$ on $\mathfrak{p}$ is a $\mathfrak{k}$-submersion metric provided that $\mathfrak{p}=\mathfrak{p}_\mathfrak{k}+\mathfrak{p}_\mathfrak{k}^\perp$ is orthogonal with respect to $g$ and that the restriction of $g$ to $\mathfrak{p}_\mathfrak{k}^\perp$ is $\Ad(K)$-invariant.

Note that in this language,  $T_e(K/H)=\mathfrak{p}_\mathfrak{k}$ and $T_e(G/K)=\mathfrak{p}_\mathfrak{k}^\perp$.
The orthogonality assumption and $\Ad(K)$-invariance imply that the homogeneous fibration \begin{equation}
K/H\to G/H\to G/K \label{fibration-intro}
\end{equation}
is a Riemannian submersion, where the induced metric on $G/K$ is $G$-invariant (see \cite{besse2007einstein} p. 257). As in \cite{besse2007einstein}, such a submersion, in addition, has totally geodesic fibers. 

If we start with a $\mathfrak{k}$-submersion metric, scale the metric on the fiber by $t$, and normalize volume to be $1$, we obtain a ``divergent'' path of metrics in $\mathcal{M}^G_1$, whose scalar curvature is given by the formula
\begin{equation}
t^{f/n}\left(  \frac{1}{t}S(K/H)+S(G/K)-t||A||^2 \right),
\label{scale-fibers}
\end{equation}
where $f=\dim(K/H)$, $S(K/H)$ is the scalar curvature of $g$ restricted to $K/H$, $S(G/K)$ is the scalar curvature of the metric induced by $g$ on $G/K$, and $||A||$ is the norm of the O'Neill tensor computed with respect to $g$ (see \cite{besse2007einstein} p. 253).  Hence if $K/H$ is not a torus (and $g$ is chosen so that $S(K/H)>0$), then $S\to \infty$ as $t\to 0$. On the other hand, if $K/H$ is a torus, then $S\to 0$ as $t\to 0$ (see \cite{wang1986existence}, \cite{bohm2004variational}). Conversely, the existence of a path with $S\to\infty$, or a path with $S\to 0$ and $|\Ric^0|\to 0$, implies the existence of such subgroups (\cite{wang1986existence},\cite{bohm2004variational}).

We now describe the Ricci flow on the two-parameter families of homogeneous metrics on spheres as studied in \cite{buzano2014ricci} and \cite{bakas2012ancient}. When the volume is normalized, they become one-parameter families, and the normalized flow can be understood in terms of the gradient flow for the single-variable function $S(t)$ on $\M^G_1$, where $t$ is the scale of the fiber in the Hopf fibration (\ref{Hopf}). Figure $2$ depicts the graph of $S$ as a function of $t$.  
\begin{figure}[h]
	\centering{
		\resizebox{150mm}{!}{\includegraphics{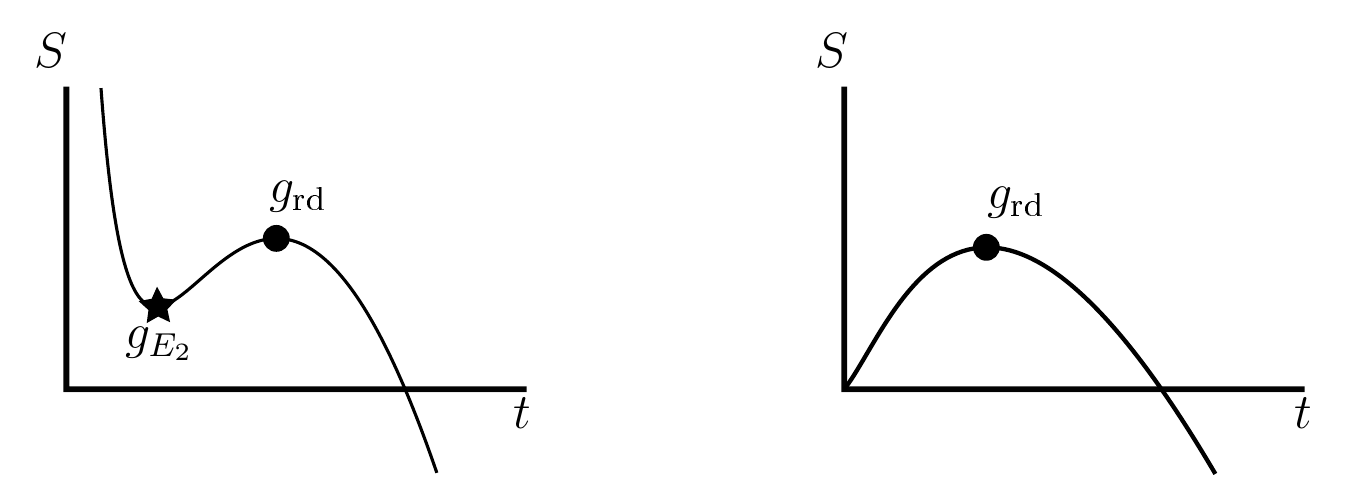}}
		\caption{Scalar curvature of two-parameter families of metrics as fiber is scaled by $t$. The graph on the left describes $\Spin(9)$ and $\Sp(n+1)\times\Sp(1)$-invariant metrics, and the graph on the right describes $\uU(n+1)$-invariant metrics.}
		\label{fig: two-parameter families}
	}
\end{figure}

For the graph on the left,  there are exactly two $G$-invariant Einstein metrics, the round metric, which we denote by $g_{\text{rd}}$, and a second Einstein metric (see \cite{jensen1973einstein},\cite{bourguignon1978curvature}), which we denote in each case by $g_{E_2}$ (although these are not isometric for different $G$). It follows from our remarks above that there are exactly two ancient solutions $\tg_t$, both converging to $g_{E_2}$ under the backwards flow.

For the graph on the right, every solution converges to $g_{\text{rd}}$. There is one solution $\tg_s$ with $S\to 0$ as the $S^1$ fiber shrinks to a point under the backwards flow, and hence by (\ref{gradientflow}) this solution is ancient. In \cite{bakas2012ancient}, it was shown in both cases that these conclusions also hold for the Ricci flow, although the arguments are more involved. Note also that in Theorem \ref{ancient for ricci implies ancient normalized}, we prove that a solution to the Ricci flow is ancient if and only if the corresponding solution for the normalized flow is also ancient.

For the case of left-invariant metrics on $S^3$,
in \cite{isenberg1992ricci} the authors showed that the normalized flow always converges to the round metric. In \cite{cao2008backward}, the authors further proved that the only ancient solutions for the Ricci flow are the Berger spheres, i.e., metrics satisfying $x\le y=z$ where $x,y$ and $z$ are the eigenvalues of the metric.

The only remaining family of homogeneous metrics on spheres are the $\Sp(n+1)$-invariant metrics, which we now describe.

\subsection{$\Sp(n+1)$-invariant metrics on spheres}

We view $S^{4n+3}$ as the unit sphere in $\mathbb{H}^{n+1}=\mathbb{R}^{4n+4}$ with the standard Euclidean product. The group of quaternionic-linear isometries $\text{Sp}(n+1)$ acts transitively on $S^{4n+3}$ with stabilizer $\text{Sp}(n)$ at the point $p=(1,0,...,0)$, so that $S^{4n+3}\simeq\Sp(n+1)/\Sp(n)$. With respect to the $\text{Ad(Sp}(n+1))$-invariant inner product $Q$ on $\mathfrak{sp}(n+1)$, $Q(A,B)=-\frac{1}{2}\Re(\text{trace}(AB))$, we have the orthogonal decompositions $\mathfrak{sp}(n+1)=\mathfrak{sp}(n)\oplus\mathfrak{p}$ and $\mathfrak{p}=\fp_0\oplus\fp_1$, where $\fp_0\simeq\mathfrak{sp}(1)=\langle i,j,k\rangle$ is the Lie algebra embedded diagonally as
\[
\begin{pmatrix}
\mathfrak{sp}(1)
& \rvline & 0 \\
\hline
0 & \rvline &
0
\end{pmatrix},
\]
and $\fp_1\simeq \mathbb{H}^n$ via
$$X\mapsto 
\begin{pmatrix}
0
& \rvline & -\overline{X}^t \\
\hline
X & \rvline &
0
\end{pmatrix}.
$$
Notice that $i,j,k$ have length $\frac{1}{2}$ and are $Q$-orthogonal. The representation of $\Ad(\Sp(n))$ on $\mathfrak{p}$ is trivial on $\fp_0$ and acts by usual matrix multiplication on $\fp_1$. Hence, by Schur's lemma any $\Ad(\Sp(n))$ invariant inner product on $\mathfrak{p}$ is of the form $\sigma+s\langle \cdot ,\cdot \rangle_{\mathbb{H}^n}$ where $\langle\cdot,\cdot\rangle_{\mathbb{H}^n}$ is the Euclidean inner product on $\mathbb{H}^n$ and $\sigma$ is any inner product on $\fp_0\simeq\mathbb{R}^3$. As observed in \cite{ziller1982homogeneous}, the normalizer $N(\Sp(n))/\Sp(n)=\text{SO}(3)$ acts on $\fp_0$ by $Q$-isometries, so we can diagonalize $\sigma$ with respect to the $Q$-orthogonal basis $\langle i,j,k\rangle$.



Henceforth, we write an $\text{Sp}(n+1)$-invariant metric on $S^{4n+3}$ as \begin{equation}g=x\langle,\rangle|_{(i)}+y\langle,\rangle|_{(j)}+z\langle,\rangle|_{(k)}+ s\langle,\rangle|_{\mathbb{H}^n}, \label{metrics}
\end{equation} where $\langle,\rangle$ is the standard metric on $\mathbb{H}^{n+1}=\mathbb{R}^{4n+4}$. Note also that via right translations, $N(H)/H=\uSO(3)$ acts by diffeomorphisms on $G/H$ fixing $H$ and induces the usual linear action of $\uSO(3)$ on $\fp_0=\RR^3$. In particular, there exist global diffeomorphisms of $S^{4n+3}$ which  switch the signs of $i,j,k\in T_{p}S^{4n+3}$, two at a time. These are only isometries if the metric is of the form (\ref{metrics}). Since isometries are always preserved under the Ricci flow, metrics of the form (\ref{metrics}) are preserved as well.


 For studying the Ricci flow, it will often be convenient to consider the Ricci endomorphism. We denote the Ricci endomorphism by $\text{ric}$ and the Ricci curvature tensor by $\Ric$, i.e., $\Ric(X,Y)=g(\text{ric}(X),Y)$.

For the above metrics, the Ricci endomorphism decomposes as $$
\ric=r_i\Id|_{(i)}+r_j\Id|_{(j)}+r_k\Id|_{(k)}+r_h\Id|_{\HH^n},	
$$
where
\begin{align}
r_i &=2\left(\frac{x^2-y^2-z^2}{xyz}\right)+\frac{4}{x}+\frac{4nx}{s^2}\label{ric}\nonumber\\
r_j &=2\left(\frac{y^2-x^2-z^2}{xyz}\right)+\frac{4}{y}+\frac{4ny}{s^2} \\
r_k &=2\left(\frac{z^2-x^2-y^2}{xyz}\right)+\frac{4}{z}+\frac{4nz}{s^2}\nonumber \\
r_h &=-2\left(\frac{x+y+z}{s^2}\right)+\frac{4 n+8}{s} \nonumber
\end{align}
(see \cite{ziller1982homogeneous}).
Thus, the scalar curvature is given by the formula \begin{equation}
S=\frac{4}{x}+\frac{4}{y}+\frac{4}{z}+\frac{16n(n+2)}{s}-4n\left(\frac{x+y+z}{s^2}\right)-2\left(\frac{x^2+y^2+z^2}{xyz}\right). \label{Scalar xyzs}
\end{equation}

As in \cite{ziller1982homogeneous}, for any $T\in\mathbb{H}^n$ we have the sectional curvatures $K(i,T)=\frac{x}{s^2}$, $K(j,T)=\frac{y}{s^2}$, and $K(k,T)=\frac{z}{s^2}$. From this and the fact that an isometry preserves eigenspaces of the Ricci tensor, it is not difficult to see that there are no further isometries among metrics of the form (\ref{metrics}), besides permuting the variables $x,y$ and $z$.


Our work is closely related to the examples of homogeneous Einstein metrics on spheres and projective spaces. These were classified by Ziller in \cite{ziller1982homogeneous} and can be obtained by scaling the fibers in the Hopf fibrations
$$S^1\to S^{4n+3}\to \mathbb{CP}^{2n+1} \quad\text{       }\quad S^3\to S^{4n+3}\to \mathbb{HP}^n \quad\text{       }\quad S^7\to S^{15 }\to S^8.$$

In \cite{ziller1982homogeneous} it was shown that the only $\Sp(n+1)$-invariant Einstein metrics on $S^{4n+3}$ are, up to scaling, the round metric $g_{\text{rd}}$, given by $x=y=z=s=1$, and Jensen's second Einstein metric $g_{E_2}$, given by $x=y=z=1$ and $s=2n+3$.

If we view $\langle i\rangle=\fu(1)$ as tangent to the Hopf action, then $\Sp(n+1)$-invariant metrics on $\mathbb{CP}^{2n+1}$ are precisely the metrics induced by $\fu(1)$-submersion metrics on $S^{4n+3}$. Since $\uU(1)\subset N(H)$ acts by fixing $i$ and rotating the $j,k$ plane, $\fu(1)$-submersion metrics on $S^{4n+3}$ satisfy $y=z$ and are hence of the form \begin{equation}
g=x\langle,\rangle|_{(i)}+y\langle,\rangle|_{(j)}+y\langle,\rangle|_{(k)}+ s\langle,\rangle|_{\mathbb{H}^n}.
\end{equation}
On $\mathbb{CP}^{2n+1}$, these induce metrics of the form $$
y\langle,\rangle|_{(j)}+y\langle,\rangle|_{(k)}+ s\langle,\rangle|_{\mathbb{H}^n}, \label{CPn metrics}
$$
and Ziller showed that the only two Einstein metrics in this family are given by $y=s$ (the Fubini-Study metric), which we denote  by $g_{\mathbb{CP}^{2n+1}}^{\text{FS}}$, and
 $y/s=1/(n+1)$, which we denote by $g_{\mathbb{CP}^{2n+1}}^2$. %
 Note that metrics of this form can be obtained by scaling the fibers and base of the Hopf fibration $S^2\to\CCPP^{2n+1}\to\mathbb{HP}^n$ as in \cite{ziller1982homogeneous}.

\section{General Results}

In this section, we discuss some results that hold for all compact homogeneous spaces, relating properties of solutions of the Ricci flow to those of the normalized flow. We first prove that, as in the case of the Ricci flow, the normalized flow develops a singularity in finite time, unless it converges to an Einstein metric. Our proof is similar to the proof of Theorem 4.1 in \cite{bohm2015long}.
\begin{thm}
Let $\tg_t$ be a solution to the normalized Ricci flow on a compact homogeneous space. 
Then $\tilde{T}_{\text{max}}=\infty$ if and only if $\tg_t$ converges to an Einstein metric. Furthermore, if $\tilde{T}_{\text{max}}<\infty$ then $S(\tg_t)\to\infty$ as $t\to\tilde{T}_{\text{max}}$.

 \label{thm: finite extinction}
\end{thm}

\begin{proof}
B\"{o}hm showed that on a compact homogeneous space that is not a torus, the Ricci flow develops a Type-1 singularity in finite time \cite{bohm2015long}, and hence, by results of 
Naber (\cite{naber2010noncompact}), Enders, M\"{u}ller, Topping (\cite{enders2011type}), Petersen and Wylie (\cite{petersen2009gradient}), along any sequence of times $t_i\to T_{\text{max}}$, a parabolic sequence of rescaled solutions $$g_i(t):=S(g(t_i)) g\left(t_i+\frac{t}{S(g(t_i))}\right)$$ subconverges to a soliton $g_\infty(t)$ on $E^k_\infty\times\RR^{n-k}$, where $E^k_{\infty}$ is a compact homogeneous Einstein manifold and $\RR^{n-k}$ is endowed with the flat metric. Furthermore, the dimension of the Euclidean factor depends only on the initial metric, not on the sequence $t_i$ \cite{bohm2015long}. 
As we will see, the presence of a Euclidean factor in the limit determines whether or not the normalized flow converges to an Einstein metric. We then use the dimension of the Euclidean factor to control the growth of $S$ near the extinction time.

Since the normalized flow is the $L^2$ gradient flow for $S$ we can derive an evolution equation for $S$ under the normalized flow:
 \begin{align*}
S(\tg_t)'=\langle \nabla S,\nabla S\rangle_{L^2}
&=\left\langle-2\left(\Ric(\tg_t)-\frac{S(\tg_t)}{n}\tg_t\right),-2\left(\Ric(\tg_t)-\frac{S(\tg_t)}{n}\tg_t)\right)\right\rangle_{L^2}\\
&=4\left(|\Ric(\tg_t)|^2-\frac{S^2(\tg_t)}{n}\right).		
\end{align*}

	Let $\bar{g}(t)=S(g(t))g(t)$. We claim that the eigenvalues of $\Ric(\bar{g}(t))$ all converge to $r_\infty=\frac{1}{k}$ or $0$. If not, there would exist a $\delta>0$ and a sequence of times $t_i$ such that $\Ric(\bar{g}(t_i))=\Ric(g_i(0))$ has an eigenvalue in $(-\infty,-\delta)\cup(\delta,r_\infty-\delta)\cup(r_{\infty}+\delta,\infty)$. But then the same would be true for any subsequence of $g_i(0)$ and hence also for the limit soliton, which is a contradiction since $S(g_\infty(0))=\lim_{i\to\infty}S(\bar{g}(t_i))=1$ and the dimension of the Einstein factor depends only on the initial metric.

Let $r_1(t),...,r_k(t)$ be the eigenvalues of $\Ric(\bar{g}(t))$ which converge to $r_\infty$. Then for $t$ sufficiently close to $T_{\text{max}}$,
 \begin{align*}|\Ric(g(t))|^2-\frac{S^2(g(t))}{n}&=S^2(g(t))\left(|\Ric(\bar{g}(t))|^2-\frac{1}{n}\right) \\
 &\ge S^2(g(t))\left(\sum_{i=1}^k r_i^2(t)-\epsilon -\frac{1}{n}\right)\\
 &\ge S^2(g(t))\left(k r_{\infty}^2-2\epsilon-\frac{1}{n}\right)=S^2(g(t))\left( \frac{1}{k}-\epsilon-\frac{1}{n}  \right)
 \end{align*}
 Since both sides of the above inequality scale the same way, the same is true for $\tg_t$. Hence if $t$ is sufficiently large and $k<n$, $\frac{\partial S(\tg_t)}{\partial t}\ge C S^2(\tg_t)$, which implies that $S(\tg_t)\to\infty$ in finite time. 

	On the other hand, if the dimension of the Euclidean factor in the limit is zero, then $\bar{g}(t)$ converges to a $G$-invariant Einstein metric, and hence the volume-normalized solution converges to an Einstein metric and $\tilde{T}_{\text{max}}=\infty$.
	

	
\end{proof}

For the main example of our paper, we would also like to classify the ancient solutions for the Ricci flow on $\M^G_1$. Before doing so we first show that on homogeneous spaces, being ancient for the Ricci flow is equivalent to being ancient for the normalized flow. 
\begin{thm}
	A solution $g_t$ of the Ricci flow on a compact homogeneous space is ancient if and only if the corresponding solution $\tg_t$ of the normalized flow is also ancient. Furthermore, if $\tg_t$ is ancient and does not converge to an Einstein metric as $t\to-\infty$, then $S(\tg_t)\to 0$ and $|\Ric^0(\tg_t)|\to 0$ as $t\to-\infty$ and hence $\tg_t$ is a $0$-Palais Smale solution.
	\label{ancient for ricci implies ancient normalized}
\end{thm}
\begin{proof}
If $g_t$ is ancient then since $S(\tg_t)>0$ for all $t$ and since $\tg_t$ is the gradient flow for $S$, $\tg_t$ is ancient as well (see Section 1).

	In order to prove that if $\tg_t$ is ancient then $g_t$ is also ancient, we first prove that an ancient solution of the normalized flow on a compact homogeneous space has positive scalar curvature. Since a theorem of Lafuente (see \cite{lafuente2015scalar}) states that a homogeneous solution of the Ricci flow with finite backwards singular-time must have $S\to-\infty$ as $t\to T_{\text{min}}$,  it follows that $g_t$ must be ancient as well.
	
	Now, suppose we have an ancient solution to the normalized flow with $S(\tg(0))\le 0$. Then, since $G/H$ is not a torus, and hence is not Ricci flat, $S(\tg(t))<0$ for all $t<0$. 
	By Bochner's theorem, the Ricci tensor of a compact homogeneous space has at least one positive eigenvalue. In particular, if $\{r_i\}_{i=1}^n$ are the eigenvalues of $\Ric$ and $r_1,...,r_{n-k}$ are all the positive eigenvalues (where $k\le n-1$), then by Cauchy-Schwarz and the fact that $S<0$, $$|S|^2\le \left|\sum_{i=n-k+1}^n r_i\right|^2\le k\sum_{i=n-k+1}^n r_i^2\le k|\Ric|^2.$$
	
	For the backwards flow, the evolution equation for $S$ is $\frac{\partial S}{\partial t}=4\left(\frac{S^2}{n}-|\Ric|^2\right)$, and hence $$\frac{\partial S}{\partial t}=
	4\left(\frac{S^2}{n}-|\Ric|^2\right)\le 4\left( \frac{k}{n}|\Ric|^2-|\Ric|^2 \right)\le \frac{4(k-n)}{n}|\Ric|^2\le \frac{4(k-n)}{n^2}S^2
	.$$
Thus $\frac{\partial S}{\partial t}=-C S^2$ for some $C>0$ and hence $S(t)\to-\infty$ in finite time, contradicting the assumption that $\tg_t$ is ancient.
\end{proof}

\section{Ricci flow on Spheres}

We now study the Ricci flow of $\Sp(n+1)$-invariant metrics on spheres. Recall that metrics of the form (\ref{metrics}) are preserved under the Ricci flow. We can thus view $x,y,z,$ and $s$ as functions of time.

Recall also that the normalized Ricci flow on $\M^G_1$ is given by the ODE $$\frac{\partial	}{\partial t}\tg_t=-2\left(\Ric -\frac{S(\tg_t)}{\dim(M)}\tg_t\right)$$ and hence satisfies
	\begin{align}
	 x'&=-2x\left(r_i-\frac{S}{4n+3} \right)\label{riccifloweqns}\nonumber \\
   y'&=-2y\left(r_j-\frac{S}{4n+3} \right)\\
   z'&=-2z\left(r_j-\frac{S}{4n+3}\right)\nonumber \\
   s'&=-2s\left(r_h-\frac{S}{4n+3}\right)\nonumber, 
 \end{align}
 where $r_i,r_j,r_k,r_h$ and $S$ are as in (\ref{ric}) and (\ref{Scalar xyzs}).
Since the normalized Ricci flow preserves volume, we can restrict these ODE's to $\M^G_1$, the space of volume-1 metrics, i.e. those satisfying $xyzs^{4n}=1$.	
We parametrize these metrics by setting $s=\frac{1}{(xyz)^{\frac{1}{4n}}}$. Hence the normalized Ricci flow is equivalent to an ODE in $\mathbb{R}^3_{>0}$.
For later convenience, we include the formula for scalar curvature in the above coordinates:
\begin{equation}
S=\frac{4}{x}+\frac{4}{y}+\frac{4}{z}-\frac{2 z}{x y}-\frac{2 y}{x
   z}-\frac{2 x}{y z}+16n(n+2) (x y z)^{\frac{1}{4n}}-4 n (x+y+z) (x y
   z)^{\frac{1}{2n}}.
   \label{eqn:scalarxyz}
\end{equation}

	\begin{lem}
		The metrics where two of the variables agree, i.e. $x=y$, $y=z$, or $x=z$ are precisely those which are $\Sp(n+1)\times \text{U}(1)$-invariant, and the metrics where $x=y=z$ are precisely those which are $\Sp(n+1)\times\Sp(1)$-invariant, and hence these metrics are preserved by the Ricci flow. The metrics where $y=z=s$ are $\uU(2n+2)$-invariant, and hence also preserved by the Ricci flow.
		\label{U(1)}
	\end{lem}
\begin{proof}
If $K\subset\Sp(1)$, we view the action of $\Sp(n+1)\times K$ on $S^{4n+3}$ as 
 $(g,k)\cdot p= g p k^{-1}$. Note that since $p=(1,0,...,0)$ is totally real, $kpk^{-1}=p$ for all $k\in K$. 
 Invariant metrics under this larger group can then be viewed as the subset of $G$-invariant metrics that are also invariant under the adjoint action of $K\subset\Sp(1)$ on $\fp_0=\mathfrak{sp}(1)$. If $K=\uU(1)=\{e^{i\theta}\}_{\theta\in[0,2\pi)}$ for example, then a metric is invariant if and only if rotation in the $j,k$ plane is an isometry, and hence if and only if $y=z$. If $K=\Sp(1)$ then a metric is invariant if and only if its restriction to $\mathfrak{sp}(1)$ is a multiple of the bi-invariant metric and hence if and only if $x=y=z$.
 
The action of $\uU(2n+2)$ on $S^{4n+3}$ is by isometries if and only if the adjoint action of the stabilizer $\uU(2n+1)$ acting on $\langle p,ip\rangle ^\perp=(jp)\oplus (kp)\oplus\mathbb{H}^n$ is by isometries. But this is the case if and only if the metric on $(jp)\oplus (kp)\oplus\mathbb{H}^n$ is a multiple of the Euclidean metric, i.e., if $y=z=s$.
\end{proof}
\begin{lem}
	The only fixed points of the normalized flow are the round metric, where $x=y=z=1$ and Jensen's second Einstein metric $x=y=z=(2n+3)^{-\frac{4n}{4n+3}}$. The round metric is a stable node
	  and Jensen's second Einstein metric is a saddle node.
	 The tangent space of the unstable manifold at Jensen's metric is given by $x=y=z$, and the tangent space of the stable manifold is given by $x+y+z=3(2n+3)^{-\frac{4n}{4n+3}}$. \label{local-behaviour}
\end{lem}	

\begin{proof}
By symmetry, the Hessian of	the system when $x=y=z$ must be of the form $$ \begin{bmatrix}
	a & b &  b\\
	b & a &  b\\
	b & b & a
\end{bmatrix},$$
which has eigenvectors $(1,1,1), (2,-1,-1),$ and $(-1,2,-1)$ with corresponding eigenvalues $a+2b$, $a-b$ and $a-b$. A direct calculation shows that at $x=y=z=(2n+3)^{-\frac{4n}{4n+3}} $,  
\begin{align*}
a&=-8  \left(2 n^2+7 n+5\right)(2 n+3)^{-\frac{4 n+6}{4 n+3}}\\
b&=16 (n+1) (n+2) (2 n+3)^{-\frac{4 n+6}{4 n+3}}.
\end{align*}

Thus $a-b<0$ and $a+2b>0$.

At the round metric, a direct calculation shows
\begin{align*}
a=-8 (1 + n)\quad\text{and}\quad
b=0,
\end{align*}
and hence the round metric is a stable node.

\end{proof}

By symmetry in the three variables, it suffices to understand the Ricci flow on the set $$\Omega=\{(x,y,z)\in\mathbb{R}^3:0< x\le y\le z\},$$ which is preserved by the Ricci flow since the boundary consists of invariant sets (see Lemma \ref{U(1)}). 

As in Section 1, there are two possibilities for the long time behavior of $\tg_t$. Either $\tg_t$ converges to an Einstein metric or $S(\tg_t)\to\infty$ as $t\to\tilde{T}_{\text{max}}<\infty$. If $S(\tg_t)\to\infty$, then since the Ricci flow preserves metrics of the form (\ref{metrics}), we can apply Theorem 4.6 in \cite{bohm2015long} (see also Remark 5.4 on p. 557), which in this case implies that $S(\tg_t)\tg_t$ converges to an isometric product $S^3\times\mathbb{R}^{4n}$ where $S^3$ and $\mathbb{R}^{4n}$ are endowed with the round metric and flat metric, respectively. We offer an elementary proof along with a monotonicity lemma that is useful for our classification of ancient solutions. 
\begin{thm}
Any $\Sp(n+1)$-invariant solution to the normalized Ricci flow on $S^{4n+3}$ either converges to the round metric, Jensen's second Einstein metric, or $S\to\infty$ in such a way that $S(\tg(t))\tg(t)$ converges in the pointed $C^\infty$ topology to $S^3\times\mathbb{R}^{4n}$, where $S^3$ and $\RR^{4n}$ are endowed with the round metric and flat metric respectively. Furthermore, in the last case, $x,y,z\to 0$ in such a way that $x/z$ and $y/z$ monotonically converge to $1$.  \label{forward flow limits}
\end{thm}
\begin{proof}

The structure of our proof is as follows. First, we will show that the ratios $x/z$ and $y/z$ are monotonic along any solution of the normalized flow. Then we will see that $S\to \infty$ only if all the variables go to zero. Lastly, we will show that when all the variables tend to zero, their ratios tend to $1$.

\begin{lem}  \label{lem: non-increasing}
For any solution of the normalized Ricci flow in $\Omega$, the ratios $y/z$ and $x/z$ are non-decreasing. In particular, if $z\to 0$ as $t\to \tilde{T}_{\text{max}}$, then $x\to 0$ and $y\to 0$ as well.
\end{lem}
\begin{proof}
The derivative of $(y/z)$ is given by \begin{equation}
\left(\frac{y}{z}\right)'=\frac{y'z-z'y}{z^2}=(yz)\frac{-r_j+r_k}{z^2}.
\end{equation}
In particular, $(y/z)'$ has the same sign as 
\begin{equation}
r_k-r_i=\frac{-4 (y-z) \left(n (x y z)^{\frac{1}{2 n}+1}-x+y+z\right)}{x y
	z}.
\end{equation}
Since we assumed $x\le y\le z$, we see that $(y/z)'\ge 0$, with equality if and only if $y=z$. Similarly, $x/z$ is also non-decreasing.

\end{proof}

Assume from now on that $\tg_t$ does not converge to an Einstein metric, i.e., that $S(\tg_t)\to\infty$. Since $S$ is continuous on $\Omega$ the only way for $S\to \infty$ is if $\tg_t$ approaches the boundary of $\Omega$, that is if $x\to 0$, or if $z\to \infty$. 

If $z\to \infty$ and $x$ remains bounded away from zero, then the only way we can have $S\to \infty$ in (\ref{eqn:scalarxyz}) is if $(x y z)\to\infty$. However, if $(x y z)$ is sufficiently large then $(x y z)^\frac{1}{2n}>(x y z)^\frac{1}{4n}$, in which case the $-4nx(xyz)^\frac{1}{2n}$ term dominates all the positive terms. Thus, we can assume that $z\to 0$ and hence also that $x\to 0$ and $y\to 0$ as well by Lemma \ref{lem: non-increasing}.

Since the ratios $x/z$ and $y/z$ are less than or equal to $1$ and non-decreasing they each converge to some finite positive constant. Let $\lim_{t\to \tilde{T}_{\text{max}}}x/z=C$, $\lim_{t\to  \tilde{T}_{\text{max}}}y/z=D$. Since the ratio $x/z$ is scale invariant, the limit is the same for the Ricci flow and the normalized flow. Suppose for the moment that for the Ricci flow, $\lim_{t\to T_{\text{max}}}x'/z'$ and $ \lim_{t\to T_{\text{max}}}y'/z'$ 
exist, and hence that$$
\lim_{t\to T_{\text{max}}}\frac{x}{z}=\lim_{t\to T_{\text{max}}}\frac{x'}{z'}=\lim_{t\to T_{\text{max}}}\frac{x}{z}\frac{r_i}{r_k}
$$
Since $\lim_{t\to T_{\text{max}}}\frac{x}{z}$ is some positive constant, this implies $\lim_{t\to T_{\text{max}}}\frac{r_i}{r_k}=1$. The same reasoning implies $\lim_{t\to T_{\text{max}}}\frac{r_j}{r_k}=1$ as well. 
A quick computation shows that as $s\to\infty$ the ratio $r_i/r_k$ tends to $-(x + y - z)/(x-y-z)$,  and hence $(-x-y+z)/(x-y-z)\to 1$ as $t\to \tilde{T}_{\text{max}}$. From this and the corresponding limit for the other quotient it follows that $x/z\to 1$ and $y/z\to 1$, and hence also $x/y\to 1$.



Now we show the limits $\lim_{t\to T_{\text{max}}}x'/z'$ and $\lim_{t\to T_{\text{max}} }y'/z' $ exist. We remark, that for the Ricci flow as well $\lim_{t\to T_{\text{max}}}x/s=\lim_{t\to T_{\text{max}}}z/s=0$ since this is true for the normalized flow.

From (\ref{ric}),
\begin{align*}
xr_i&=2\left(\frac{x^2-y^2-z^2}{yz}\right)+4+4n\frac{x^2}{s^2}\to 2C^2 D^{-1} D-2D- 2D^{-1}+4\\
z r_k &=2\left(\frac{z^2-x^2-y^2}{xy}\right)+4+4n\frac{z^2}{s^2}\to
2C^{-1}D^{-1}-2CD^{-1}-2C^{-1}D+4,
\end{align*}
and hence $\lim_{t\to T_{\text{max}}}\frac{x'}{z'}=\lim_{t\to T_{\text{max}}}\frac{x}{z}\frac{r_i}{r_k}$ exists. The calculation for the other ratios is similar.

One can see from the formula (\ref{Scalar xyzs}) that $S(\tg_t)\to \infty$ at a rate of $6/x$ and hence the 
eigenvalues of $S(\tg_t)\tg_t$ tend to $x=y=z=6$ and $s=\infty$. It follows from the Cheeger-Gromoll Splitting theorem that the limit metric splits as an isometric product $S^3\times \mathbb{R}^{4n}$. To see this, notice that the limit has non-negative Ricci curvature since the eigenvalues $\bar{r}_i,\bar{r}_j,\bar{r}_k\to 2 $ and $\bar{r}_h\to 0$ where $\bar{r}_i,\bar{r}_j,\bar{r}_k,\bar{r}_h$ are the eigenvalues of the Ricci tensor of $S(\tg_t)\tg_t$. Moreover, from the fact that the projection $S^{4n+3}\to\mathbb{HP}^n$ is a Riemannian submersion and the scale $s$ of the $\mathbb{HP}^n$ base goes to infinity, it follows that in the limit, any geodesic with initial velocity $v\in\mathbb{H}^n=\mathbb{R}^{4n}$ is minimizing for all time, and hence we can split off a line for each $v\in\mathbb{R}^{4n}$. On the remaining 3-dimensional manifold $M^3$, the Ricci curvature is constant, and hence so is the sectional curvature. In particular, $M^3$ is covered by $S^3$. However, for $t$ near $\tilde{T}_{\text{max}}$, any geodesic tangent to the fiber is closed with length close to $12\pi$ since the eigenvalues $x,y,z$ are close to $6$ and the $S^3$ fibers are totally geodesic. Hence the limit has no short closed geodesics, and it follows that $M^3=S^3$.

\end{proof}

  By definition of the stable manifold, all metrics in it converge to $g_{E_2}$. Also recall that stable manifolds are always smooth manifolds (see e.g. \cite{brin2002introduction} p.122).
\begin{thm}
	The stable manifold  for the second Einstein metric separates the space of metrics into two connected components, namely into the set of metrics which converge to the round metric and the set of metrics where $S\to\infty$ under the normalized flow. \label{separate}
\end{thm}

\begin{proof}

	First we prove that the set of metrics in $\M^G_1$ with $S\to\infty$ under the normalized flow is open. Recall that the normalized flow is the $L^2$ gradient flow for $S$ on $\M^G_1$. By \cite{bohm2004variational}, the set of Einstein metrics in $\mathcal{M}_1^G$ is compact, and hence has bounded scalar curvature, say, by $S_0$. Now let $g(t)$ be a solution of the normalized flow with $S(g(t))\to\infty$. Then there exists a time $t_0$ such that $S(g(t_0))>S_0$. By continuous dependence on initial conditions, there is an open set $U$ around $g(0)$ so that for every metric $h\in U$ the solution $h(t)$ of the normalized flow with $h(0)=h$ satisfies $S(h(t_0))>S_0$.
	But by Palais-Smale, if the scalar curvature of a solution $h(t)$ surpasses $S_0$, then in fact $S(h(t))\to\infty$. 

	Now, recall that any solution $\tg_t$ either converges to an Einstein metric, or has $S(\tg_t)\to\infty$ in finite time. Let $\gamma(t)$ be a path in $\mathcal{M}_1^G$ with $\gamma(0)$ converging to the round metric and $\gamma(1)$ a metric with $S\to\infty$ under the normalized flow. For each $t\in[0,1]$ define $F(t)\in\mathbb{R}\cup\{\infty\}$ so that $[0,F(t))$ is the maximal interval of existence of the normalized flow with initial condition $\gamma(t)$. Note that $F(0)=\infty$ and $F(1)$ is finite by Theorem \ref{thm: finite extinction}. Let $t_0=\inf(\{t\in[0,1]\sst F(t)=\infty \})$. Then we claim $\gamma(t_0)$ must lie in the stable manifold of the second Einstein metric. On the one hand, $\gamma(t_0)$ cannot converge to the round metric, since, as an attractor, the set of metrics converging to the round metric is open. On the other hand, since the set of metrics with finite extinction time for the normalized flow is open, $F(\gamma(t_0))=\infty$ (finite extinction time is equivalent to $S(\tg_t)\to\infty$). In particular, the solution of the normalized flow with initial condition $\gamma(t_0)$ must converge to an Einstein metric which is not the round metric, and hence must converge to the second Einstein metric. 
\end{proof}

\section{Ancient Solutions}
We now turn to classifying the ancient solutions for the Ricci flow in $\M^G$. Recall that given an ancient solution $g_t$ there are two possibilities as $t\to-\infty$ for the corresponding normalized solution $\tg_t$ in $\M^G_1$. Either $\tg_t$ converges to an Einstein metric or $S(\tg_t)\to 0$ and $|\Ric^0(\tg_t)|\to 0$, i.e., $\tg_t$ is $0$-Palais Smale. In \cite{pediconi2019diverging}, Pediconi proved that a 0-Palais-Smale sequence asymptotically approaches a submersion metric for a homogeneous fibration $K/H\to G/H\to G/K$ where $K$ is some intermediate subgroup with $K/H$ is a torus. We will use this result together with our monotonicity results to argue that any such solution is actually a submersion metric \textit{for all time} with respect to the Hopf fibration $\uU(1)\to S^{4n+3}\to\CCPP^{2n+1}$. Besides these solutions there are two more ancient solutions which converge to $g_{E_2}$ as $t\to-\infty$. These arise by starting with the round metric and scaling the fibers and base of the Hopf fibration $S^3\to S^{4n+3}\to\mathbb{HP}^n$ (see Section 1).

By Lemma \ref{U(1)} we can assume, up to isometry, that our solutions $\tg_t$ satisfy $x\le y\le z$ for all $t$.

\begin{lem}
	Let $\tg_t$  be an ancient solution for the normalized flow with $S(\tg_t)\to 0$ as $t\to-\infty$.  Then for all $t\in (-\infty,T_{\text{max}}) $, $y=z$. In particular, $\tg_t$ is invariant under the larger group $\uU(1)\Sp(n+1)$.
\end{lem}

\begin{proof}
Since $\lim_{t\to-\infty}|\Ric^0(\tg_t)|=0$, it follows that for any sequence of times $t_i\to-\infty$ we must have $\tilde{g}(t_i)\to \infty$ in $\mathcal{M}^G_1$, otherwise there would exist a subsequence converging to a flat metric, contrary to our assumption.
	 
	Each metric in $\M^G_1$ can be written uniquely in the form $$g_{v,s}:=e^{sv_1}\langle,\rangle_{(i)}+e^{sv_2}\langle,\rangle_{(j)}+e^{sv_3}\langle,\rangle_{(k)}+e^{sv_4}\langle,\rangle_{\mathbb{H}^n}$$ where $v_1^2+v_2^2+v_3^2+v_4^2=1$ and $v_1+v_2+v_3+v_4=0$. 
	 	 
	Define the sequences $v^{(i)}\in S^3$, $s^{(i)}\in\RR$ by $\tilde g(t_i)=g_{v^{(i)},s^{(i)}}$. Then, since $S^3$ is compact, there exists a subsequence $v^{(i)}\to v^{(\infty)}$ and $s^{(i)}\to\infty$. By Theorem 4.1 in \cite{pediconi2019diverging}, $v^{(\infty)}$ is a so-called submersion direction for some toral $H$-subalgebra $\mathfrak{k}$, that is, a subalgebra $\mathfrak{k}=\text{Lie}(K)$ where $K$ is connected, $H\subset K\subset G$, and the quotient $K/H$ is a torus.  Moreover, $\mathfrak{k}\cap\mathfrak{p}$ is generated by the $\Ad(H)$-irreducible summands of $\mathfrak{p}$ corresponding to the most shrinking eigenvalue. By Proposition 3.10 in \cite{pediconi2019diverging}, if $v$ is a submersion direction for an $H$-subalgebra $\mathfrak{k}$, then $g_{v,s}$ is a $\mathfrak{k}$-submersion metric for all $s\in\RR$ and moving along the path $\gamma_v(s)=g_{v,s}$ is equivalent to shrinking the fibers of the homogeneous fibration $K/H\to G/H\to G/K$.
	
	In our case $v^{(\infty)}$ is a submersion direction for some toral $\Sp(n)$-subalgebra. On the other hand, the only toral subalgebras containing $\fsp(n)$ are isomorphic to $\fsp(n)\oplus \fu(1)$, where $\fu(1)$ is the Lie algebra of some circle subgroup of $\Sp(1)$ (the only Lie subgroups of $\Sp(n+1)$ containing $\Sp(n)$ are isomorphic to $\Sp(n)$, $\uU(1)\Sp(n)$ and $\Sp(1)\Sp(n)$). Since we assumed $x\le y \le z$, it follows that $x$ is the most shrinking eigenvalue, and hence $\mathfrak{k}=\mathfrak{sp}(n)\oplus(i)$.
	
	Let $(v_1^{(\infty)},v_2^{(\infty)},v_3^{(\infty)},v_4^{(\infty)})$ be the components of $v^{(\infty)}$. Since $\gamma_{v^{(\infty)}}(s)$ is invariant under the larger isometry group $\uU(1)\Sp(n+1)$ where $\uU(1)=\{e^{i\theta}\}_{\theta\in[0,2\pi)}\subset \Sp(1)$, we can conclude that $v_2^{(\infty)}=v_3^{(\infty)}$ and $v_1^{(\infty)}<0$.
	Moreover, since $Q([j,k],i)=1\ne 0$, Theorem 4.1 in \cite{pediconi2019diverging} further implies that $y(t_i)/z(t_i)\to 1$ as $t_i\to -\infty$. 	
	
		 On the other hand, by Lemma \ref{lem: non-increasing}, along the backwards flow $y/z$ is non-increasing. Hence $\lim_{t\to-\infty}y/z$ exists and equals $1$. But again, since $y/z$ is non-increasing and $y/z\le 1$, this is only possible if $y=z$ for all $t$.
		 \end{proof}

Hence for the purpose of classifying ancient solutions, it suffices to consider metrics in $\M^G_1 $ of the form \begin{equation}
\frac{1}{y^2s^{4n}}\langle,\rangle_{(i)}+y\langle,\rangle_{(j)}+y\langle,\rangle_{(k)}+s\langle,\rangle_{\mathbb{H}^n}. \label{ancient-form}
\end{equation}

Moreover, referring to the above proof, since $v_1^{(\infty)}<0$, we can assume that $\frac{1}{y^2 s^{4n}}\to 0$ for ancient solutions that do not converge to an Einstein metric as $t\to-\infty$. Notice that in this section our normalization differs from the one in Section 3. For metrics of the form (\ref{ancient-form}), the scalar curvature is given by \begin{equation}
S=\frac{16 n^2}{s}-\frac{8 n y}{s^2}-\frac{1}{s^{4n}y^2} \left(\frac{4
	n}{s^2}+\frac{2}{y^2}\right)+\frac{32 n}{s}+\frac{8}{y}.
\label{scalar-curvature-ancient-form}
\end{equation}

We prove the following classification result.
\begin{thm}
	Let $\tg_0=g_{x,y,z,s}$ with $x\le y\le z$. Then $\tg_t$ is ancient if and only if $x\le y=z\le s$. \label{ancient-classification}
\end{thm}

Note that metrics with $y=\frac{1}{y^2 s^{4n}}$ are precisely the ones invariant under the larger group of isometries $\Sp(1)\Sp(n+1)$, and hence these converge to $g_{E_2}$ as $t\to-\infty$. Metrics with $y=s$ are invariant under the group $\uU(2n+2)$ by Lemma \ref{U(1)} and are hence also preserved. These two solutions were shown to be ancient in \cite{bakas2012ancient}.
We begin with a lemma.
\begin{lem}
	 For any ancient solution with $\lim_{t\to-\infty}S(\tg_t)=0$, the ratio $y/s$ remains bounded as $t\to-\infty$. Moreover, if $\lim_{t\to-\infty}y/s$ exists and is non-zero, then the only possibilities are $\lim_{t\to-\infty}y/s=1$ or $\frac{1}{1+n}$.
 \label{possible limits}
\end{lem}

\begin{proof}

We can bound (\ref{scalar-curvature-ancient-form}) above by 
\begin{equation}
S<\frac{8n}{s}\left(2n+4 -\frac{y}{s}\right)+\frac{8}{y}
\label{scalar-curvature-ancient-form-2}
\end{equation}

If $y/s>M$ then we can bound (\ref{scalar-curvature-ancient-form-2}) above by
$$
\left(8n(2n+4-M)+\frac{8}{M}\right)\frac{1}{s}
$$
which is negative if $M$ is sufficiently large. But ancient solutions of the Ricci flow (and hence also of the normalized flow) have non-negative scalar curvature, and hence $y/s$ must be bounded.

Now we examine the possible limits for $y/s$ as $t\to-\infty$. Since $\frac{1}{y^2 s^{4n}}\to 0$ and $y/s$ is bounded as $t\to-\infty$, $s\to\infty$ as well.

Suppose that $\lim_{t\to-\infty} \frac{y}{s}=C>0$. Then since $s\to\infty$, $y\to\infty$ as well. For the Ricci flow, 
$$
\lim_{t\to-\infty}\frac{y'}{s'}=\lim_{t\to-\infty}\frac{y}{s}\frac{r_j}{r_h}.
	$$
From (\ref{ric}) we have
	\begin{equation*}
 y r_j=-\frac{2x}{y}+4+\frac{4ny^2}{s^2}
\quad\text{        and        }\quad s r_h=8+4n-\frac{2x}{s}-\frac{4y}{s}.
\end{equation*}

Since $x/y\to 0$ and $x/s\to 0$ under the normalized flow, the same is true for the Ricci flow. Thus, since $y,s\to\infty$ and since we assumed $\lim_{t\to-\infty}\frac{y}{s}=C$, both of  the above quantities tend to finite limits, and hence $\lim_{t\to-\infty}\frac{y'}{s'}$ exists. But then also
$$
C=\lim_{t\to-\infty}\frac{y}{s}=\lim_{t\to-\infty}\frac{y'}{s'}=\frac{4+4nC^2}{8+4n-4C}
.$$
Solving the above equation yields $C=1$ or $C=\frac{1}{1+n}$. Since the ratio $y/s$ is scale-invariant, the same holds for the normalized flow.

\end{proof}

Note that these two ratios correspond to the two homogeneous Einstein metrics $g^{\text{FS}}_{\mathbb{CP}^{2n+1}}$ and $g^{2}_{\mathbb{CP}^{2n+1}}$ on the base $\CCPP^{2n+1}$ (see Section $1$ and \cite{ziller1982homogeneous}). 
\begin{lem}
Solutions with $y/s>1$ are not ancient.
\end{lem}	
\begin{proof}
	We will show that if $y/s>1$ then $(y/s)'>0$ under the backwards flow. But since $y/s$ is bounded for any ancient solution, $y/s$ would converge to a finite limit greater than $1$, which would contradict the previous lemma.

	The derivative of $y/s$ under the backwards flow is $$
	(y s)\frac{r_j-r_h}{s^2},
	$$
	which has the same sign as \begin{align}r_j-r_h&=-\frac{8+4n}{s}-\frac{2}{y^4 s^{4 n}}+\frac{2}{y^2 s^{2+4n}}+\frac{4}{y}+\frac{(4+4n) y}{s^2}\\
	&=\frac{2 (y-s) \left(s+y+2 s^{4 n} y^3 ((1+n)y-s)\right)}{y^4 s^{2+4n}}. \label{rj-rh}	\end{align}

	which is positive since $y>s$.
\end{proof}

Now to conclude the proof of Theorem \ref{ancient-classification} we only need to show that the remaining solutions are ancient.
\begin{lem}
Solutions satisfying $\frac{1}{y^2 s^{4n}}\le y\le s$ are ancient for the normalized flow. Furthermore, along such a solution $y,s\to\infty$ and either $y=s$ or $y/s\to \frac{1}{1+n}$.
 \label{they are ancient}
\end{lem}
\begin{proof}
We already saw that metrics satisfying $y=\frac{1}{y^2s^{4n}}$ or $y=s$ are preserved by the Ricci flow and are ancient. Hence, solutions which begin in the set $\Gamma=\{\frac{1}{y^2s^{4n}}\le y\le s\}$ remain in that set, and we can assume from now on $\frac{1}{y^2 s^{4n}}<y<s$.

Now, we prove that $s\to\infty$ for any solution $\tg_t$ in the interior of $\Gamma$. 
Assume $\tg_t$ is not ancient. Then one of the variables must go to $0$ or $\infty$ as $t\to \tilde{T}_{\text{min}}$. We will show that in each possible scenario, $s\to\infty$.
If $s\to 0$ then since $y/s<1$, $y\to 0$ as well, but this contradicts $y^3>\frac{1}{s^{4n}}$. If $y\to \infty$, then $y/s<1$ implies $s\to \infty $ as well. If $y\to 0$ then $y^3>\frac{1}{s^{4n}}$ implies $s\to\infty$.  

Next, we look at the derivative of $y/s$ under the backwards flow, which, as before, has the same sign as (\ref{rj-rh}), except now $y-s<0$, since $\tg_t$ is in the interior of $\Gamma$. Hence it has the same sign as \begin{equation}
\label{(y/s)' backwards showing its ancient}
-s-y+2s^{4n}y^{3}(s-(1+n)y)
.\end{equation}

Since $s\to \infty$, we see further, that for fixed $y/s$, and for large enough $s$, (\ref{(y/s)' backwards showing its ancient}) is positive if $y/s<\frac{1}{1+n}$ and negative if $y/s\ge\frac{1}{1+n}$. Hence $y/s$ does not return to the same value infinitely many times. 
But this implies $y/s\to \frac{1}{1+n}$, for if $y/s$ crosses $\frac{1}{1+n}$, then one can argue that $y/s$ is eventually contained in any neighborhood of $\frac{1}{1+n}$. If $y/s>\frac{1}{1+n}$ for all time, then $y/s$ must converge to $\inf_{t\in(\tilde{T}_{\text{min}},\tilde{T}_{\text{max}})}\{(y/s)(t)\}$, and hence by Lemma \ref{possible limits}, must converge to $\frac{1}{1+n}$, and similarly if $y/s<\frac{1}{1+n}$ for all time. Thus $y/s\to\frac{1}{1+n}$ and both $y,s\to\infty$.

From the formula for the scalar curvature (\ref{scalar-curvature-ancient-form}), it follows that $S\to 0$, and, in particular, $S$ is bounded from below, and thus the solution is ancient.




\end{proof}

\begin{prop}
Solutions satisfying $\frac{1}{y^2s^{4n}}<y\le s$ are collapsed. Moreover, if $\frac{1}{y^2s^{4n}}<y<s$ then a rescaling of $\tg_t$ converges in the Gromov-Hausdorff sense to $g^{2}_{\mathbb{CP}^{2n+1}}$ as $t\to -\infty$, and if $y=s$ then a rescaling of $\tg_t$ converges to  $g^{\text{FS}}_{\mathbb{CP}^{2n+1}}$ as $t\to\infty$.
\end{prop}
\begin{proof}
By Lemma \ref{they are ancient}, such solutions satisfy $y,s\to \infty$, and $\lim_{t\to-\infty}y/s=1$ or $\lim_{t\to-\infty}y/s=\frac{1}{1+n}$. From equations (\ref{ric}), it follows that the eigenvalues of the Ricci tensor decay at a rate of $O(\frac{1}{y})$ as $t\to-\infty$, and hence $|\Ric(\tg_t)|$ decays at a rate of $O(\frac{1}{y})$. By the Gap theorem \cite{bohm2019optimal}, this implies $|\Rm(\tg_t)|$ decays at a rate of $O(\frac{1}{y})$ as $t\to-\infty$ as well. Thus the length of $i$ goes to zero for the curvature normalized solution $|\Rm(\tg_t)|\tg_t$. Since the $\uU(1)$ fibers are totally-geodesic, this implies the injectivity radius tends to zero, and hence $\tg_t$ is collapsed. In the proof of Lemma \ref{they are ancient}, we showed that if $y<s$ then $\frac{y}{s}\to\frac{1}{1+n}$. In particular, for every solution in our $1$-parameter family (besides the one with $y=s$), the metric
on the base tends to the second Einstein metric on $\CCPP^{2n+1}$ (see \cite{ziller1982homogeneous}).

\end{proof}

 \begin{figure}[h]
	\centering{
		\resizebox{90mm}{!}{\includegraphics{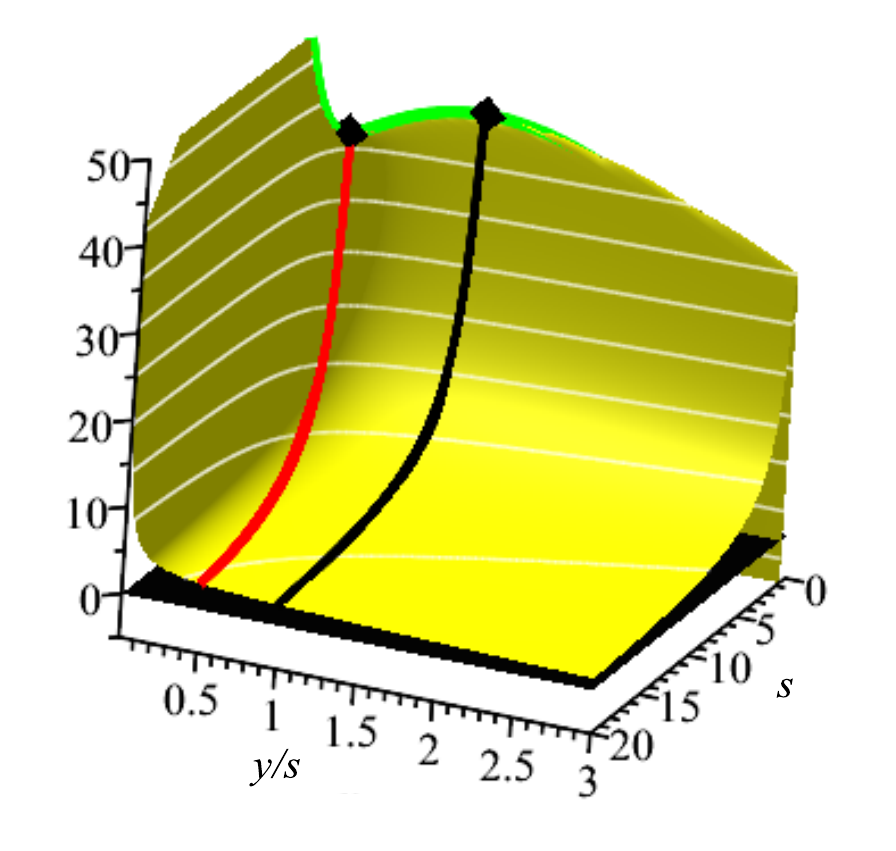}}
		\caption{The graph of $S$ over metrics with $y=z$ when $n=1$ (compare with Figure \ref{fig: ancient4}). The green line represents the $\Sp(n+1)\Sp(1)$-invariant metrics, and the black line represents the line $y=s$, or the $\uU(2n+2)$-invariant metrics. 
		The diamonds represent the round metric, which is a local maximum, and Jensen's second Einstein metric, which is a saddle point. Ancient solutions with $y<s$ asymptotically approach the red line, which represents the stable manifold for Jensen's second Einstein metric. 
\label{Scal1}}}
\end{figure}

\bibliographystyle{alpha}
\bibliography{References}
\end{document}